\documentclass[twoside,12pt,leqno]{article}
\usepackage{amsmath, amsthm, amsfonts, amssymb, color}
\usepackage{mathrsfs}
\usepackage{fancyhdr}

\usepackage[active]{srcltx}
\usepackage{enumerate}
\usepackage{amsopn} % provides \DeclareMathOperator
\usepackage{bbm} % additional fonts
\usepackage{mathrsfs}

\newcommand{\be}{\begin{eqnarray}}
\newcommand{\ee}{\end{eqnarray}}
\newcommand{\ce}{\begin{eqnarray*}}
\newcommand{\de}{\end{eqnarray*}}

\newtheorem{thm}{Theorem}[section]

\newtheorem{lem}[thm]{Lemma}
\newtheorem{prp}[thm]{Proposition}

\theoremstyle{definition}
\newtheorem{defn}{Definition}[section]

\definecolor{wco}{rgb}{0.5,0.2,0.3}

\numberwithin{equation}{section}
\theoremstyle{remark}
\newtheorem{rem}{Remark}[section]

\newcommand*{\norm}[1]{\lVert#1\rVert}

\def\R{\mathbb{R}}

\def\<{\langle} \def\>{\rangle}  
    
\def\d{\text{\rm{d}}}   \def\D{\mathcal D}
  
 \def\beq{\begin{equation}}  \def\F{\mathcal F}

 \def\ee{\varepsilon}

\def\l{\lambda}
\def\L{\Lambda}

\def\A{\mathcal{A}}
\def\D{\Delta}

\def\V{\mathcal{V}}
\def\VV{\mathcal{V}^*}

\def\[{{\Big[}}
\def\]{{\Big]}}

\def\({{\Big(}}
\def\){{\Big)}}

\DeclareMathOperator*{\esssup}{ess\,sup}

\title{{\bf   Quasi-Linear (Stochastic) Partial Differential Equations with Time-Fractional Derivatives}
\footnote{Financial support by the DFG
through the CRC ``Taming uncertainty and profiting from randomness and
low regularity in analysis, stochastics and their applications'' is
acknowledged.  W.L. is  supported by NSFC (No. 11571147),  NSF of Jiangsu Province
 (No. BK20160004), the Qing Lan Project and
PAPD of Jiangsu Higher Education Institutions, J.L.S. is supported by project  I$\&$D: UID/MAT/04674/2013.}
}
\author{ {\bf  Wei Liu$^{a}$\footnote{Corresponding author: weiliu@jsnu.edu.cn}
, Michael R\"{o}ckner$^{b}$, Jos\'{e} Lu\'{i}s da Silva$^{c}$
}\\
 \footnotesize{  $a.$ School of Mathematical Sciences, Jiangsu Normal University, 221116 Xuzhou, China}\\
 \footnotesize{  $b.$  Fakult\"at f\"ur Mathematik, Universit\"at Bielefeld,
 D-33501 Bielefeld, Germany}\\
\footnotesize{$c.$ CIMA, University of Madeira, Campus da Penteada, 9020-105 Funchal, Portugal
}\\
}
\date{}
\begin{document}
\maketitle

\begin{abstract}
In this paper we develop a method to solve (stochastic) evolution equations on Gelfand triples with time-fractional derivative based on monotonicity techniques.
Applications include deterministic and stochastic quasi-linear partial differential equations with time-fractional derivatives,
including time-fractional (stochastic) porous media equations (including the case where the Laplace operator is also fractional)
and $p$-Laplace equations as special cases.

\end{abstract}

\bigbreak

\section{Introduction}

Fractional Calculus has a long history and its origins can be traced back to the end of seventeenth century (cf.~\cite{R77}).
Although the first main steps of the theory date back to the first half of the nineteenth century, this subject became
very active only over the last few decades. One of the main reasons is that scientists  and engineers
have developed many new models that naturally involve fractional differential equations, which have been applied successfully, e.g.~in mechanics (theory of viscoelasticity and viscoplasticity)(cf.~\cite{M10}), bio-chemistry (modelling of polymers and proteins)(cf.~\cite{DE88, F80}), electrical
engineering (transmission of ultrasound waves)(cf.~\cite{Di10}), medicine (modelling of human
tissue under mechanical loads)(cf.~\cite{CRA08,Ma10}) etc. For more applications and references we refer to \cite{BDS,He11,MS12,MK04}.

A typical example of a time-fractional equation is the time-fractional heat equation $\partial_t^\beta u = \Delta u$ with $0 < \beta < 1$, where $\partial_t^\beta$
is the Caputo fractional
derivative first appeared in \cite{Ca67} and is defined for
 $0 < \beta < 1$ by
\begin{equation}\label{Caputo}
  \partial_t^\beta  f(t)=\frac{1}{\Gamma(1-\beta)} \frac{d}{dt} \int_0^t (t-s)^{-\beta}(f(s)-f(0)) ds,
\end{equation}
where $\Gamma$ is the Gamma function defined by $\Gamma(\lambda) :=\int_0^\infty t^{\lambda-1} e^{-t} dt$. Here one implicitly assumes that $f$ is such that the integral on the right-hand side is (weakly) differentiable in $t$.
For the precise domain of $f$ for which $\partial_t^\beta f$ is defined and which is convenient in our case, we refer to \eqref{eq:2.9'} below.

The classical heat equation $\partial_t u = \Delta u$ describes heat propagation in homogeneous medium.
The time-fractional diffusion equation $\partial_t^\beta u = \Delta u$ with $0 < \beta < 1$ has been widely used to model
 anomalous diffusions exhibiting subdiffusive behavior, e.g. due to particle sticking and trapping phenomena (cf.~\cite{MS12,SKW82}).
While in normal diffusions (described by the
heat equation or more general parabolic equations), the mean squared displacement
of a diffusive particle behaves like const$\cdot t$ for $t\rightarrow\infty$, the time-fractional diffusion equation
exhibits a behaviour like const$\cdot t^\beta$ for $t\rightarrow\infty$.
 This is the reason why time-fractional equations with  $0 < \beta < 1$  are
called subdiffusion equations in the literature, and for the case  $1 < \beta < 2$  are
called superdiffusion equations. In fact, there is a diverse number of
real world systems which demonstrate this type of phenomena. For example,
the above time-fractional equations and nonlinear variants of them are also widely used to model dynamical processes
in materials with memory, e.g.~the diffusion
of fluids in porous media with memory (see \cite{C99}). We refer to the survey article \cite{MK} and the monographs \cite{BDS,Mur,P99,P93,SKM,Zh14} for more references.

We also want to mention the interesting interplay between stochastic processes and time-fractional differential equations.
A celebrated result of Einstein established a mathematical link between random walks, the diffusion equation and Brownian motion.
The probability densities of the Brownian motion solve a diffusion equation (heat equation), and hence we refer to the Brownian motion as the stochastic solution to
the heat equation.
Similarly, the stochastic solutions of time-fractional diffusion equations are subordinated processes (e.g.~iterated Brownian motion).
The solutions to fractional diffusion equations are strictly related to stable densities.
Indeed, the stochastic solutions of time-fractional diffusions can be realized through ``time-change'' by inverse stable subordinators.
Just as Brownian motion is a scaling limit of simple random walks, the stochastic solutions to certain fractional diffusion equations are scaling limits of continuous time random walks, in which the $i.i.d.$ jumps are separated by $i.i.d.$ waiting times.
For this subject we refer to \cite{BMN09,MNV09,OB09} and the references therein.

In recent years, there has been also growing interest in stochastic time-fractional partial differential equations (see e.g.~\cite{CY97, SRY97}).
For example, the author in \cite{Zh10} considers stochastic Volterra equations with singular kernels in a 2-smooth Banach space.  The authors in \cite{CKK} study the $L^2$-theory for a class of semilinear SPDEs with time-fractional
 derivatives, which can be used to describe
random effects on transport of particles in media with thermal memory, or particles subject to sticking and
trapping.
In  \cite{MN15,MN16}, the authors consider a space-time fractional stochastic heat type equation to model phenomenon with random
effects with thermal memory, and they prove existence and uniqueness of mild solutions and also some intermittency property (see also \cite{FN15}).
For a linear stochastic partial differential equation of fractional order both in the time and space variables with a different type of noise term  we refer to \cite{CHHH,HH}.
In \cite{C14,CHN} the authors investigate linear stochastic time-fractional partial differential equations for the more general case $0<\beta \le 2$.

In this paper, we are mainly interested in non-linear stochastic time-fractional partial differential equations in a variational setting,
which include in particular important examples of quasi-linear type as the stochastic porous media or the $p$-Laplace equation.
They are of the following form:
\begin{equation}\label{SDE}
 \partial_t^{\beta}X(t)+A(t,X(t))=\partial_t^\gamma \int_0^t B(s) dW(s),  \ 0<t<T,  % \ u(0)=u_0\in H,
\end{equation}
where $V \subseteq H\equiv H^*\subseteq V^*$ is a Gelfand triple, $\beta\in(0,1], \gamma\in(0,\beta+\frac{1}{2})\cap (0, 1]$, $\partial_t^{\beta}$ is the Caputo  fractional derivative, $A:[0,T]\times V\rightarrow V^*$ and $B: [0,T]\rightarrow L_{HS}(U; H)$ (here $L_{HS}(U; H)$
denotes the space of all Hilbert--Schmidt operators from $U$ to $H$) are measurable. For simplicity we only consider the additive type noise in (\ref{SDE}) here,
the case of general multiplicative noise will be investigated in a forthcoming paper.

We will establish the existence and uniqueness of solutions to (\ref{SDE}) under classical monotonicity and coercivity conditions on $A$ (see Theorems \ref{Th1} and \ref{Th2}).
The proof of the main results will be based on a functional analytic approach, which, to the best of our knowledge, is new to solve
both the deterministic and stochastic time-fractional partial differential equations in the variational setting.
This approach is inspired by the work of Stannat \cite{St99} on the theory of generalized Dirichlet forms, which in turn draws essential ideas from \cite{Li69}.
We will establish a general existence result concerning an abstract operator equation (see Theorem \ref{T1} below) which is central to our approach.

We should mention that  time-fractional linear evolution equations in the Gelfand triple setting have first been investigated in \cite{Za09}. Later on the author
also proved the global solvability of a non-degenerate parabolic equation with time-fractional derivative in \cite{Za12} (cf. \cite{ACV16} for more general case).
However, these results cannot be applied to
quasilinear type equations like the porous media or the $p$-Laplace equation. In \cite{JP04} the authors investigate
elliptic-parabolic integro-differential equations with
$L^1$-data. Their framework includes the time-fractional $p$-Laplace equation. However, the authors in \cite{JP04}  consider generalized solutions
($i.e.$ entropy solutions). Therefore, the results of this paper generalize or complement  the corresponding results in \cite{JP04,Za09,Za12}
in the setting of time-fractional quasilinear PDE with monotone coefficients. Recently,
the authors in \cite{VZ15} derive very interesting decay estimates for the solutions of
time-fractional porous medium and $p$-Laplace equations (by assuming the existence), and
the decay behaviour is notably different from the case with usual time derivative. Hence, here we, in particular,
give a positive answer to the question of the existence and
uniqueness of solutions to the time-fractional porous medium equations and $p$-Laplace equations, which are left open in \cite{VZ15}.

  If $\beta=\gamma=1$, then both the deterministic and stochastic equation (\ref{SDE})  have been intensively investigated in the variational setting.
  For the deterministic case, one might refer to the monographs \cite{Ba10,Br73,Li69,Sh97,Z90} and the references therein.
  But also the stochastic case (SPDE) has attracted more and more attention in recent years,
  we refer to some classical references  \cite{BT73,KR79,Par75} and recent works \cite{BR15,BLZ,G,L13,LR10,LR13,LR15,RRW07,Zh09} (see also the references therein).

  The paper
is organized as follows. In Section 2 we present the main results on the existence and uniqueness of solutions to
deterministic and stochastic nonlinear evolution equations with time-fractional derivatives.
In Section 3 we will apply our main results to stochastic time-fractional porous medium equations and $p$-Laplace equations
as model examples. The well-posedness of both equations have been open problems even in the deterministic case.

\section{Main Results}
Let  $(H, \<\cdot,\cdot\>_H)$ be a real separable
Hilbert space identified with its dual space $H^*$ by the Riesz
isomorphism. Let   $V$ be  a real reflexive  Banach space, continuously and densely embedded into $H$. Then we have the
 following Gelfand triple
$$V \subseteq H\equiv H^*\subseteq V^*.$$
Let ${ }_{V^*}\<\cdot,\cdot\>_V$ denote the
 dualization
between  $V$ and its dual space $V^*$. Then it is easy to show that
$$ { }_{V^*}\<u, v\>_V=\<u, v\>_H, \ \  u\in H ,v\in V.$$

Now, for $T \in [0, \infty)$ fixed, we consider the following general nonlinear evolution equation with time-fractional derivative
\begin{equation}\label{1.1}
 \partial_t^{\beta} (u(t)-x)+A(t,u(t))=f(t),  \ 0<t<T,
\end{equation}
where
$$\partial_t^{\beta}(u-x):= \frac{1}{\Gamma(1-\beta)} \frac{d}{dt} \int_0^t (t-s)^{-\beta}(u(s)-x) ds, $$
is the  Riemann--Liouville fractional derivative, which coincides with the Caputo  fractional derivative  if  $u(0)=x$ ($i.e.$ considered $x$
as the initial condition),
$A:[0,\infty) \times V\rightarrow V^*$ is restrictedly measurable, $i.e.$  for  each $dt$-version of
$u\in L^1([0,\infty); V)$, $t\mapsto A(t,u(t))$ is $V^*$-measurable on $[0,\infty)$, and $f\in L^1([0,\infty); V^*)$.

Applying the  Riemann--Liouville fractional integral $I_t^\beta$, defined by
$$ I_t^\beta f(t):=\frac{1}{\Gamma(\beta)}\int_0^t(t-s)^{\beta-1}f(s) ds =: (g_\beta * f)(t),$$
to equation (\ref{1.1}) one obtains that
\begin{equation}\label{IE}
 u(t)= x - \frac{1}{\Gamma(\beta)}\int_0^t (t-s)^{\beta-1} A(s,u(s)) \, ds +  \frac{1}{\Gamma(\beta)} \int_0^t (t-s)^{\beta-1}f(s) \, ds,
\end{equation}
for $dt$-a.e.~$t \in [0, \infty)$.

Here we recall that $\frac{d}{dt}$ in the definition of $\partial^{\beta}_t$  (see (\ref{Caputo})) is understood as a weak derivative.
Defining $\tilde{u}$ to be equal to the right-hand side of \eqref{IE}, then this $dt$-version
of $u$ satisfies $\tilde{u} (0) = x$. Concerning the continuity properties of $\tilde{u}$ we refer to Theorem 2.1 below.
We also remark that below we work with functions $u$ only determined $dt$-a.e.,
so $u (0)$ can be chosen arbitrarily. Therefore,  we write $x$ in \eqref{1.1} instead of $u(0)$.

Now let us formulate the precise conditions on the coefficients in (\ref{1.1}).
Suppose for  fixed $\alpha>1$ that there exist constants $\delta>0$, $C$ and $g\in L^1([0,\infty);  \mathbb{R}_{+})$ such that the following conditions hold for all $t\in[0,\infty)$ and $v,v_1,v_2\in V$.
\begin{enumerate}
 \item [$(H1)$] (Hemicontinuity)
      The map  $ s\mapsto { }_{V^*}\<A(t,v_1+s v_2),v\>_V$ is  continuous on $\mathbb{R}$.

\item[$(H2)$] (Monotonicity)
     $$  { }_{V^*}\<A(t,v_1)-A(t, v_2), v_1-v_2\>_V
     \ge 0. $$

\item [$(H3)$] (Coercivity)
    $$  { }_{V^*}\<A(t,v), v\>_V  \ge  \delta
    \|v\|_V^{\alpha}  - g(t).$$

\item[$(H4)$] (Growth)
   $$ \|A(t,v)\|_{V^*} \le  g(t)^{\frac{\alpha-1}{\alpha}} +
   C\|v\|_V^{\alpha-1} .$$
\end{enumerate}

We define the following  spaces,
\begin{align*}
  \mathcal{V}&=L^\alpha([0,\infty); V)\cap L^2([0,\infty); H),\\
  \mathcal{H}&=L^2([0,\infty); H),\\
  \mathcal{V}^*&=L^{\frac{\alpha}{\alpha-1}}([0,\infty); V^*)+L^2([0,\infty); H),
\end{align*}
where
$\| \cdot \|_{\mathcal V} := \max ( \| \cdot \|_{L^\alpha([0,\infty); V)} , \| \cdot \|_{\mathcal H} )$ and for $u \in \mathcal V^*$
$$\| u \|_{\mathcal{V}^*} := \inf \left\{ \| u_1 \|_{L^{\frac{\alpha}{\alpha-1}}([0,\infty); V^*)} + \| u_2 \|_{\mathcal H}:
u_1 \in L^{\frac{\alpha}{\alpha-1}}([0, \infty); V^*), u_2 \in \mathcal H \text{ such that } u=u_1+u_2\right\}.$$

Then for $x=0$ (the case for general initial condition $x$ will then follow easily as we shall see below)
the original equation (\ref{1.1}) can be rewritten in the following form
\begin{equation}\label{1.1'}
 \partial_t^{\beta}u+\mathcal{A}u=f,
\end{equation}
where
$$\mathcal{A}: \mathcal{V}\rightarrow \mathcal{V}^*; (\mathcal{A}u)(t)=A(t,u(t)), \ t\in[0,\infty).$$

It is easy to see that $\mathcal{A}: \mathcal{V}\rightarrow \mathcal{V}^*$ is  monotone, coercive and bounded on bounded sets.

We shall see below that $\partial_t^{\beta}$ with domain $\{ u \in \mathcal H \mid r \mapsto |r|^\beta \hat u(r) \in L^2( \R ; H)\}\cap \V$ is closable as an operator from $\V$ to $\V^*$.
Let $(\partial_t^{\beta}, \F)$ denote its closure with norm $\|u\|_{\F} := \left( \|u\|_{\V}^2 + \|\partial_t^{\beta} u\|_{\V^*}^2 \right)^{\frac12} $, $u \in \F$.
Here $\hat u$ denotes the Fourier transform of $u$, considered as a function from $\R$ to $H$, setting $u \equiv 0$ on $\R \setminus (- \infty, 0)$.

\begin{thm}\label{Th1}
Suppose that $T \in [0,\infty)$ and $A \colon [0, \infty) \times V \rightarrow V^*$ satisfies $(H1)$-$(H4)$. Then for every $x \in V$ and $f \in \mathcal{V}^*$, (\ref{1.1}) has a unique solution $u$ such that $u - x \varphi \in \F$ for every $\varphi \in L^\alpha([0,\infty);  \mathbb{R})$ with $\varphi \equiv 1$ on $[0, T+1)$. In particular,
\begin{equation}\label{eq:th1-1}
  u-x\varphi\in L^\alpha([0,\infty); V); \   \partial_t^\beta (u- x \varphi) \in L^{\frac{\alpha}{\alpha-1}}([0,\infty); V^*)
\end{equation}
and for $dt$-a.e.~$t \in [0, T]$,
\begin{equation}\label{IE2}
 u(t)= x - \frac{1}{\Gamma(\beta)}\int_0^t (t-s)^{\beta-1} A(s,u(s))\, ds +  \frac{1}{\Gamma(\beta)} \int_0^t (t-s)^{\beta-1}f(s) \, ds .
\end{equation}
Furthermore, $t \mapsto \frac{1}{\Gamma(1-\beta)} \int_0^t (t-s)^{-\beta} (u(s) - x \varphi(s)) \, ds$ has a continuous $H$-valued $dt$-version,
and if $\beta \in (\frac{\alpha-1}{\alpha}, 1)$,  $t \mapsto u(t)$ has a  continuous $V^*$-valued $dt$-version.

\end{thm}

\begin{rem}\label{R1}
$(i)$ As a matter of fact, the monotonicity assumption $(H2)$ of $A$ is a sufficient condition to imply that $\mathcal{A}: \mathcal{V} \rightarrow \V^*$ is pseudo-monotone.
The above theorem still holds if we replace $(H2)$ by assuming $\mathcal{A}: \mathcal{V} \rightarrow \V^*$ is pseudo-monotone.
In \cite{LR13,LR15}, a local monotonicity condition is assumed for $A$ which yields that $A(t,\cdot): V\rightarrow V^*$ is pseudo-monotone provided $V \subseteq H$
is a compact embedding. However, it is still not clear whether it implies that $\mathcal{A}: \mathcal{V} \rightarrow \V^*$ is pseudo-monotone.

$(ii)$ It is easy to see from the proof that to have uniqueness of solutions one can replace $(H2)$ by the following weak monotonicity:
 $$  { }_{V^*}\<A(t,v_1)-A(t, v_2), v_1-v_2\>_V + K \|v_1-v_2\|_H^2
     \ge 0, $$
where $K$ is a positive constant.

$(iii)$ If $A$ is the subdifferential of a convex function, $i.e.$
$$  A(t,u)=\partial \varphi (t,u),  \  u\in V, $$
where $\varphi(t, \cdot): V \rightarrow \mathbb{R}$ is convex and continuous, then for
$$  \Phi(u):=\int_0^T \varphi(t,u(t))dt, \ u\in \mathcal{V}, $$
we have $\partial\Phi(u)=\mathcal{A}(u), \ \forall u\in \mathcal{V}. $ Then
\begin{align*}
 ~~& ~~ \mathcal{A}(u)-\Lambda u=f \\
\Leftrightarrow&~~ \partial\Phi(u)-\Lambda u=f \\
\Leftrightarrow& ~~  u=\arg\min_{v\in\mathcal{V}}\left\{ \Phi(v)+\Phi^*(f+\Lambda v)-  { }_{\V^*}\<\Lambda v, v\>_{\V} -  { }_{\V^*}\<f, v\>_{\V}  \right\},
\end{align*}
where $ \Phi^*(\eta):=\sup_{v\in \mathcal{V}}  \left\{  { }_{\V^*}\<\eta, v\>_{\V} -\Phi(v)   \right\} $.
\end{rem}
\bigbreak

Now we  recall the definition of  pseudo-monotone operator, which is  a very useful
generalization of monotone operator and was first
introduced
 by Br\'{e}zis in \cite{Br68}.  We use the  notation ``$\rightharpoonup$''
for  weak convergence in Banach spaces.

\begin{defn} The operator $M: \mathcal{V}\rightarrow \mathcal{V}^*$ is called pseudo-monotone  if $v_n\rightharpoonup v$ in $\mathcal{V}$
as $n\rightarrow \infty$ and
$$     \limsup_{n\rightarrow\infty} { }_{\mathcal{V}^*}\<M(v_n), v_n-v\>_{\mathcal{V}}\le 0 $$
implies for all $u\in \mathcal{V}$
$$   { }_{\mathcal{V}^*}\<M(v), v-u\>_{\mathcal{V}} \le  \liminf_{n\rightarrow\infty} { }_{\mathcal{V}^*}\<M(v_n), v_n-u\>_{\mathcal{V}}.  $$
 \end{defn}

\begin{rem}\label{r2.1} Browder introduced
 a slightly different definition of  a pseudo-monotone operator in \cite{Bro77}:
% \begin{definition*}
An operator $M: \mathcal{V}\rightarrow \mathcal{V}^*$ is called pseudo-monotone   if
$v_n\rightharpoonup v$ in $\mathcal{V}$ as $n\rightarrow \infty$ and
$$     \limsup_{n\rightarrow\infty} { }_{\mathcal{V}^*}\<M(v_n), v_n-v\>_{\mathcal{V}}\le 0 $$
implies
$$   M(v_n)\rightharpoonup M(v)\ \    \text{and}   \  \
\lim_{n\rightarrow\infty} { }_{\mathcal{V}^*}\<M(v_n), v_n\>_{\mathcal{V}}={ }_{\mathcal{V}^*}\<M(v), v\>_{\mathcal{V}}.  $$
% \end{definition*}
In particular, if $M$ is bounded on bounded sets, then  these two definitions are equivalent, we refer to \cite{LR13,LR15}.
\end{rem}

\begin{lem}\label{L2}
 If $ M: \mathcal{V} \rightarrow \V^*$ is pseudo-monotone, bounded on bounded sets and coercive, then $M$ is surjective,
  i.e. for any $f\in \VV$, the equation $M u=f$ has a solution.
\end{lem}
\begin{proof}
This is a classical result due to Br\'{e}zis. For the proof we refer to \cite{Br68} or \cite[Theorem 27.A]{Z90}.
\end{proof}

As in \cite{St99} we consider a generator $\Lambda$, with domain $D(\Lambda, \mathcal H)$, of a $C_0$-contraction semigroup of linear operators on $\mathcal H$ whose restrictions to $\mathcal V$ form a $C_0$-semigroup of linear operators on $\mathcal V$.
The generator of the latter is again $\Lambda$, but with domain $D(\Lambda, \mathcal V) := \{ u \in \mathcal V \cap D(\Lambda, \mathcal H)  \mid \Lambda u \in \mathcal V\}$.
Then $D(\Lambda, \mathcal V)$ is dense in $\mathcal V$, hence so is $D(\Lambda, \mathcal H) \cap \mathcal V$.

By \cite[Lemma 2.3]{St99}, $\Lambda \colon D(\Lambda, \mathcal H) \cap \mathcal V \rightarrow \mathcal V^*$ is closable as an operator from $\mathcal V$ to $\mathcal V^*$.
Denoting its closure by $(\Lambda, \mathcal F)$, $\F$ is a Banach space with norm $\| u \|_{\F} := (\|u\|^2_{\V} + \| \Lambda u\|^2_{\V^*})^{\frac12}$, $u \in \F$.

We will use the following abstract result in which $\Lambda$ will later be taken to be $-\partial_t^\beta$ to solve equation (\ref{1.1}).
This abstract result is a generalization of \cite[Proposition 3.2]{St99}.
We replace the strong monotonicity assumption in \cite[Proposition 3.2]{St99} by the classical monotonicity $(H2)$ (see also Remark \ref{R1}) and consider a reflexive Banach space $\mathcal V$ in place of the Hilbert space $\mathcal V$ in \cite{St99}.

\begin{thm}\label{T1}
 For any $f\in \VV$, there exists a solution $u\in \mathcal{F}$ of the equation $\A u- \Lambda u=f$.
\end{thm}
\begin{proof}
\textbf{Step 1}: For $\alpha>0$, consider the Yosida approximation $\Lambda_\alpha : \mathcal{V}\rightarrow \VV$ defined by
$$ \Lambda_\alpha=\alpha(\alpha V_\alpha- I),  $$
where $V_\alpha=(\alpha-\Lambda)^{-1}$, $\alpha > 0$, is the resolvent of $( \Lambda, D(\Lambda, \mathcal H))$ (on $\mathcal H$).

Note that $\A-\Lambda_\alpha$ is pseudo-monotone, coercive and bounded on bounded sets.
By Lemma \ref{L2} there exists $u_\alpha\in \V$ such that $\A u_\alpha- \Lambda_\alpha u_\alpha=f$.

\textbf{Step 2}: Note that
$$  { }_{\mathcal{V}^*}\<\A u_{\alpha}, u_{\alpha}\>_{\V}\le  { }_{\mathcal{V}^*}\<\A u_{\alpha}-\Lambda_\alpha u_\alpha, u_{\alpha}\>_{\V}= { }_{\mathcal{V}^*}\<f, u_{\alpha}\>_{\V} \le \|f\|_{\VV} \|u_\alpha\|_{\V}.  $$
Hence, by the coercivity assumption $(H3)$ we obtain that $\sup_{\alpha>0} \|u_\alpha\|_{\V} < \infty$, and hence $\sup_{\alpha>0} \|\A u_\alpha\|_{\VV} < \infty$ by $(H4)$.

Since for any $v\in \V$
 \begin{equation*}
 \begin{split}
    { }_{\mathcal{V}^*}\<\Lambda_\alpha u_{\alpha}, v\>_{\V}&=-{ }_{\mathcal{V}^*}\<\A u_\alpha-\Lambda_\alpha u_{\alpha}, v\>_{\V} + { }_{\mathcal{V}^*}\<\A u_\alpha, v\>_{\V} \\
    &=-  { }_{\mathcal{V}^*}\<f, v\>_{\V} + { }_{\mathcal{V}^*}\<\A u_\alpha, v\>_{\V} \\
                                         &\le (\|f\|_{\VV}+\|\A u_\alpha\|_{\VV})\|v\|_{\V},
 \end{split}
\end{equation*}
we have  $\sup_{\alpha>0} \|\Lambda_\alpha u_\alpha\|_{\VV} < \infty$.

By the apriori estimates above we know there exists a subsequence $\alpha_n \rightarrow \infty$ such that
\begin{equation*}
 \begin{split}
    u_{\alpha_n} &\rightharpoonup u\  \ \text{in}\  \V ; \\
    \A u_{\alpha_n} &\rightharpoonup h\  \ \text{in}\  \VV ; \\
    \Lambda_{\alpha_n}  u_{\alpha_n} &\rightharpoonup  g\  \ \text{in}\  \VV.
 \end{split}
\end{equation*}
So, it is easy to see that $h-g=f$.

Note that
 $$ \alpha_n V_{\alpha_n}   u_{\alpha_n} \rightharpoonup u\  \ \text{in}\  \V .$$
So we have
 $$ \Lambda  \alpha_n V_{\alpha_n}   u_{\alpha_n} \rightharpoonup  g\  \ \text{in}\  \VV ,$$
 since $ \Lambda  \alpha_n V_{\alpha_n}   u_{\alpha_n}=\Lambda_{\alpha_n}  u_{\alpha_n}  $. Hence $\Lambda u=g$.
 Note that $ \| \alpha_n V_{\alpha_n}   u_{\alpha_n}\|_{\V} \le 2C \|u_{\alpha_n}\|_{\V} $, hence we have
 $$ \sup_n \| \alpha_n V_{\alpha_n}   u_{\alpha_n}\|_{\mathcal{F}} <\infty,  $$
 which implies that $ u\in \mathcal{F}$. 

\textbf{Step 3}: Now we only need to show $\A u=h$.

Since  $u_{\alpha_n} \rightharpoonup u$,
 \begin{equation*}
 \begin{split}
  & \limsup_{n\rightarrow\infty} { }_{\mathcal{V}^*}\<\A u_{\alpha_n}, u_{\alpha_n}-u\>_{\V} \\
  =& \limsup_{n\rightarrow\infty} { }_{\mathcal{V}^*}\<\Lambda_{\alpha_n} u_{\alpha_n}+f, u_{\alpha_n}-u\>_{\V} \\
  =& \limsup_{n\rightarrow\infty} { }_{\mathcal{V}^*}\<\Lambda_{\alpha_n} u_{\alpha_n}, u_{\alpha_n}-u\>_{\V} \\
 \le&  { }_{\mathcal{V}^*}\<\Lambda u, u-u\>_{\V}=0    ,
 \end{split}
\end{equation*}
where the inequality follows from ${ }_{\mathcal{V}^*}\<\Lambda_{\alpha_n} ( u_{\alpha_n}-u ), u_{\alpha_n}-u\>_{\V} \le 0 $.

So, we have
$$\limsup_{n\rightarrow\infty} { }_{\mathcal{V}^*}\<\A u_{\alpha_n}, u_{\alpha_n}\>_{\V} \le  { }_{\mathcal{V}^*}\< h, u \>_{\V}. $$
Hence, by the pseudomonotonicity, we have for any $w\in \V$
 \begin{equation*}
 \begin{split}
  & { }_{\mathcal{V}^*}\<\A u, u- w\>_{\V} \\
  \le & \liminf_{n\rightarrow\infty} { }_{\mathcal{V}^*}\< \A u_{\alpha_n}, u_{\alpha_n}-w\>_{\V} \\
  \le & \liminf_{n\rightarrow\infty} { }_{\mathcal{V}^*}\<\A u_{\alpha_n}, u_{\alpha_n}\>_{\V}- { }_{\mathcal{V}^*}\< h, w\>_{\V} \\
  \le & { }_{\mathcal{V}^*}\< h, u-w\>_{\V},
 \end{split}
\end{equation*}
which implies $\A u=h$ since $w \in \V$ was arbitrary.
\end{proof}

Now we want to show that we can take $\Lambda := - \partial_t^{\beta},$ $i.e.$, we have to show that $- \partial_t^{\beta}$ generates a $C_0$-semigroup of contractions on $\mathcal{H}$ which can be restricted to a $C_0$-semigroup on $\mathcal{V}$.
To this end, let us define the following ``shift to the right'' semigroup $U_t$, $t > 0$, on $\mathcal{H}$. For $f \in \mathcal{H}$, $t \geq 0$, define
\begin{align}\label{eq:th2.3-1}
U_t f(r) := \mathbbm{1}_{[0, \infty)} (r - t) f (r-t), \ r \in [0, \infty).
\end{align}
Then it is trivial to check that $(U_t)_{t > 0}$ is a $C_0$-contraction semigroup on $\mathcal{H}$ and it obviously can be restricted to a $C_0$-semigroup on $\mathcal{V}$ (even in this case consisting also of contractions on $\mathcal{V}$).
Now fix $\beta \in (0,1)$ and define for $f \in \mathcal{H}$
\begin{align}\label{eq:th2.3-2}
 U_t^{\beta} f : = \int\limits_{0}^{\infty} U_s f \, \eta_t^{\beta} (ds), \ t \geq 0,
\end{align}
where $(\eta_t^{\beta})_{t >0}$ denotes the one-sided stable semigroup
(of probability measures on $\big([0, \infty), \mathcal{B} ([0, \infty))\big)$) of order $\beta$, $i.e.$ we have for its Fourier transform
\begin{align*}
 \hat{\eta}_t^{\beta} (s) = \int \limits_{0}^{\infty} e^{i s r} \, \eta_t^{\beta} (dr) = e^{- t(i s)^{\beta}}, \ s \in \R, \ t > 0.
\end{align*}
We note here that
\begin{align}\label{eq:th2.3-3}
 (is)^{\beta} = [\cos(\frac{\beta \pi}{2} \text{sign}(s)) +i \sin(\frac{\beta \pi}{2} \text{sign}(s)) ] |s|^\beta, \ s \in \R,
\end{align}
hence $\eta_t^\beta(ds) $ is absolutely continuous with respect to Lebesgue measure $ds$.
It is a well-known fact (see e.g. \cite[Chap.~II, Sect.~4b]{MR92}), that $(U_t^{\beta})_{t > 0}$ is also a $C_0$-semigroup of contractions on $\mathcal{H}$ which obviously can be restricted to a $C_0$-semigroup on $\mathcal{V}$ (again consisting of contractions).

\begin{prp}\label{Prop2.4}
 The generator $(\Lambda, D (\Lambda, \mathcal{H}))$ of $(U_t^{\beta})_{t > 0}$ (on $\mathcal{H}$) is given as follows
 $$D (\Lambda, \mathcal{H}) = \{ u \in \mathcal{H} \mid r \mapsto |r|^{\beta}  \hat{u}(r) \in L^2 (\R; H)\},$$
 $$(\Lambda u)^\wedge(r) = - (i r)^{\beta} \hat{u} (r), \ r \in \R,$$
 where $\hat{u}$ denotes the Fourier transform of $u$ and is considered as a function from $\R$ to $H$. In particular, if $\beta > \frac{1}{2}$, then each
 $u \in D (\Lambda, \mathcal{H})$ has a $dt$-version $\tilde{u}$ such that $[0, \infty) \ni t \mapsto \tilde{u} (t) \in H$ is continuous and $\tilde{u} (0) = 0$.
\end{prp}

\begin{proof}
 Below we consider each $\eta^{\beta}_t (ds)$ as a measure on all of $\R$, by defining
 \begin{align*}
  \eta_t^{\beta} (A) := \eta_t^{\beta} (A \cap [0, \infty)), \ A \in \mathcal{B} (\R),
 \end{align*}
 and any function $f$ in $\mathcal{H}$ as a function on $\R$ by defining $f := 0$ on $(- \infty, 0)$.

 Let $D := \{ u \in \mathcal{H} \mid r \mapsto |r|^{\beta} \hat{u} (r) \in L^2 (\R ; H)\}$.
 Then  for $u \in D$ we have
 \begin{align*}
   \frac{1}{t} (U_t^{\beta} u - u)^\wedge (r)  & = \frac{1}{t} \Big( \int_{0}^{\infty} U_s u \ \eta_t^{\beta} (ds) - u \Big)^\wedge (r)\\
   & = \frac{1}{t} \Big( \int_{\R} \mathbbm{1}_{[0, \infty)} (\cdot - s) u (\cdot - s) \ \eta_t^{\beta} (ds) - u \Big)^\wedge (r)\\
   & = \frac{1}{t} ( u * \eta_t^{\beta} - u)^\wedge (r)\\
   & = \frac{1}{t} \hat{u} (r) \Big( e^{-t (ir)^{\beta}} - 1 \Big) \xrightarrow[t \to 0]{} - (ir)^{\beta} \hat{u} (r)
  \end{align*}
 for $dr$-a.e. $r \in \R$. But since  for all $r \in \R$ and $t>0$,
 \begin{align*}
  \frac{1}{t} \Big| e^{-t (ir)^{\beta}} - 1 \Big|  \leq 2 \vert r \vert^{\beta},
 \end{align*}
 the last convergence also holds in $\mathcal{H}$. Hence $D \subset D (\Lambda, \mathcal{H})$ and
 \begin{equation}\label{eq:th2.3-4}
  (\Lambda u)^\wedge (r) = - (ir)^{\beta} \hat{u} (r), \ r \in \R.
\end{equation}
 Similarly, one checks that
 \begin{equation}\label{eq:th2.3-5}
  U_t^{\beta} D \subseteq  D  \quad \forall t > 0,
 \end{equation}
 and that $(\Lambda, D)$ is closed as an operator from $\mathcal{H}$ to $\mathcal{H}$.
 Since $D$ is dense in $\mathcal{H}$, \eqref{eq:th2.3-5} implies (see \cite[Theorem X.49]{RS75})
 that $D$ is an operator core of $(\Lambda, D (\Lambda, \mathcal{H}))$. Consequently,
 $D = D (\Lambda, \mathcal{H})$ and $\Lambda$ is given by \eqref{eq:th2.3-4}. Furthermore, we note that for all $\beta > 0$
 \begin{equation*}
  D ( \Lambda, \mathcal{H}) \subset H^{\beta} (\R; H)
 \end{equation*}
 by definition of the fractional Sobolev space $H^{\beta} (\R; H) (\subset L^2 (\R; H))$,
 which consists of (H\"older-) continuous functions on $\R$ if $\beta > \frac{1}{2}$. Therefore, if $\beta > \frac{1}{2}$, every $u \in D (\Lambda, \mathcal{H})$
 has a continuous version $\tilde{u}$ on $\R$ which is zero on $(- \infty, 0)$, hence $\tilde{u} (0) = 0$.
\end{proof}

Next we shall prove a representation formula for $\Lambda$ in terms of the Caputo derivative $\partial_t^\beta$.
To this end, we first define a domain for $\partial_t^\beta$ which is convenient in our case.
Let for $\beta \in (0,1)$
\begin{equation}\label{eq:2.9'}
D(\partial_t^\beta) := \{ u \in L^1([0,\infty); V^*) \mid g_{1-\beta} * u \in W^{1,1}([0,T],V^*) \text{ for all } T \in (0, \infty)\},
\end{equation}
where we define for $u \in L^1([0,\infty); V^*)$
$$ (g_{1-\beta} * u) (t) := \frac{1}{\Gamma(1-\beta)} \int_0^t (t-s)^{-\beta}  u(s) \, ds.$$
Obviously, for all $T \in (0, \infty)$
$$ g_{1-\beta} * u = (\mathbbm{1}_{[0,T]} g_{1-\beta}) * (\mathbbm{1}_{[0,\infty)} u) \quad \text{ on } [0,T],$$
where  for $p \in [1, \infty)$ the latter function belongs to $L^p(\R;V^*)$ if so does $\mathbbm{1}_{[0,\infty)} u$,
and is in $C(\R; V^*)$,  if in addition $\beta \in (0, \frac{p-1}{p})$.

\begin{lem}\label{lem2.5}
Let $u \in D_0 := D(\Lambda, \mathcal H) \cap \V \cap D(\partial_t^\beta) \cap L^\infty([0, \infty);H)$.
Then
\begin{equation}\label{eq:2.5-1}
  \Lambda u = - \frac{d}{dt} (g_{1-\beta} * u)
\end{equation}
and $D_0$ is an operator core for $(\Lambda, D(\Lambda, \mathcal H))$.
\end{lem}
\begin{proof}
First we note that for $T \in [0, \infty)$
\begin{equation}\label{eq:2.5-2}
\| (g_{1-\beta} * u) (t) \|_{V^*} \leq \frac{\esssup\limits_{s \in [0,T]}
\|u(s)\|_{V^*}}{\Gamma(1-\beta)(1-\beta)} t^{1-\beta} \quad \text{ for $dt$-a.e. } t \in [0, T],
\end{equation}
and the same inequality holds with $\|\cdot\|_H$ replacing $\|\cdot\|_{V^*}$.

From now on we consider all appearing functions, originally only defined on $[0, \infty)$, as functions on all of $\R$
by defining them to be equal to zero on $\R \setminus [0, \infty)$.
As in the proof of Proposition \ref{Prop2.4}, one can check that $U_t^\beta (D_0) \subset D_0$ (for which, however, it is essential that
$\eta_t^\beta(ds)$ is absolutely continuous with respect to $ds$) and that $D_0$ is dense in $\mathcal H$.
Again applying Theorem X.49 from \cite{RS75} we obtain that $D_0$ is an operator core of $(\Lambda, D(\Lambda, \mathcal H))$.
Hence it remains to prove \eqref{eq:2.5-1}.

Let us start with calculating the Laplace transform $\mathcal L$ of the right hand side of \eqref{eq:2.5-1} for any $u \in D(\partial_t^\beta) \cap L^\infty([0, \infty);V^*)$.
So let $\lambda \in (0, \infty)$.
Then integrating by parts, using \eqref{eq:2.5-2} and Fubini's Theorem we obtain
\begin{align*}
\mathcal L \left( \frac{d}{dt} (g_{1-\beta} * u) \right) (\lambda) &= \int_0^\infty e^{-\lambda t} \frac{d}{dt} (g_{1-\beta} * u) (t) \, dt
\\ &= \lim\limits_{T \rightarrow \infty} \left( e^{-\lambda T} (g_{1-\beta} * u) (T) + \lambda \int_0^T e^{-\lambda t} (g_{1-\beta} * u) (t) \, dt \right)
\\ &= \frac{\lambda}{\Gamma(1-\beta)} \int_0^\infty e^{-\lambda t} \int_0^t (t-s)^{-\beta} u(s) \, ds \, dt
\\ &= \frac{\lambda}{\Gamma(1-\beta)} \int_0^\infty \int_s^\infty (t-s)^{-\beta} e^{-\lambda t} \, dt \, u(s) \,ds
\\ &= \frac{\lambda^\beta}{\Gamma(1-\beta)} \int_0^\infty e^{-\lambda s} u(s) \,ds \, \lambda \int_0^\infty (\lambda t)^{-\beta} e^{-\lambda t} \, dt
\\ &= \lambda^\beta \mathcal L u (\lambda), \quad \lambda \in (0, \infty).
\end{align*}
For the left-hand side of \eqref{eq:2.5-1} and $u\in D_0$ we find for all $h \in H$, $\lambda \in (0, \infty)$, since $\mathcal{L} (\eta_t^{\beta}) (\lambda) = e^{- t \lambda^{\beta}}$
\begin{align*}
\left\langle \int_{0}^{\infty} \Lambda u (r) e^{- \lambda r} dr, h \right\rangle_H &= \lim \limits_{t \to 0} \frac{1}{t} \int_{0}^{\infty} \left\langle U^{\beta}_t  u(r) - u(r), h \right\rangle_H e^{- \lambda r} dr \\
 & = \lim\limits_{t \to 0} \frac{1}{t} \Big(\mathcal{L} ( \langle u,h\rangle_H * \eta_t^{\beta}) - \mathcal{L} ( \langle u,h\rangle_H) \Big) (\lambda) \\
 & = \lim\limits_{t \to 0} \frac{1}{t} (e^{-t \lambda^\beta} -1) \, \mathcal{L} ( \langle u,h\rangle_H) (\lambda)\\
 & = - \lambda^{\beta} \left\langle \mathcal{L} u (\lambda), h \right\rangle_H.
\end{align*}
Hence, $\mathcal{L} (\Lambda (u)) (\lambda) = - \lambda^{\beta} \mathcal{L} (u) (\lambda)$ and \eqref{eq:2.5-1} follows.
\end{proof}

\begin{prp}\label{Prop2.6}
 Let $u \in \mathcal{F}$. Then $u \in D (\partial_t^{\beta})$ and
 \begin{align*}
  \Lambda u = - \frac{d}{dt} (g_{1-\beta} * u).
 \end{align*}
 In particular, $g_{1-\beta} * u \in C ([0, \infty), H)$.
 \end{prp}

 \begin{proof}
  By Lemma \ref{lem2.5} there exist $u_n \in D_0$, $n \in \mathbb N$, such that as $n \to \infty$
\begin{align*}
   u_n &\longrightarrow u \text{ in } \mathcal{V}
\intertext{and}
  - \partial_t^{\beta} u_n = \Lambda u_n &\longrightarrow \Lambda u \text{ in } \mathcal{V}^*.
\end{align*}
Let $T\in (0, \infty)$. Obviously, $g_{1-\beta} * u_n \longrightarrow g_{1-\beta} * u$ in $L^{\alpha} ([0, T]; V)$ as $ n \to \infty$.
Therefore, by completeness $g_{1-\beta} * u \in W^{1, \frac{\alpha}{\alpha-1}} ([0, T]; \mathcal{V}^*)$ and
\begin{align*}
   \Lambda u = -\frac{d}{dt} (g_{1-\beta} * u) \ \text{on} \  [0, T].
\end{align*}
The last part of the assertion then follows by \cite[Theorem 1.19, pp.25]{Ba10}.
 \end{proof}

\textbf{Proof of Theorem \ref{Th1}:} Existence:
Let $\varphi$ be as in the assertion of the theorem. Define $\mathcal{A}_x$ as $\mathcal{A}$, but with
\begin{align*}
 A_x (t, v) := A (t, v + x \varphi (t)), \ t > 0, \ v \in V,
\end{align*}
replacing $A$. Then by Theorem \ref{T1} and Proposition \ref{Prop2.6} there exist $u_x \in \mathcal{F}$ such that
\begin{align*}
 \frac{d}{dt} (g_{1-\beta} * u_x) + A_x u_x = f
\end{align*}
Define $u := u_x + x \varphi$. Then $u - x \varphi (= u_x)$ satisfies \eqref{eq:th1-1} and
\begin{align*}
 \partial^{\beta}_t (u - x) + \mathcal{A} u = f \quad dt \text{-a.e.~on }  (0, T)
\end{align*}
and \eqref{1.1} is solved.
The rest assertions of the Theorem follows from the last part of the assertion of Proposition \ref{Prop2.6} and an elementary fact about convolutions.

Uniqueness: the uniqueness proof is similar to the argument in \cite[Theorem 3.1]{Za09}.
So, we only give a brief account of the proof here.

Recall that Riemann--Liouville fractional integral $I_t^\beta f:=g_\beta \ast f$ and  the definition of the Riemann--Liouville kernel
$$   g_\beta(t)=\frac{t^{\beta-1}}{\Gamma(\beta)},  \  t>0.  $$
Suppose that $u_1, u_2$ are solutions to (\ref{1.1}), then $u:=u_1-u_2$ is a solution to the following equation
$$ \partial_t^\beta u + A(t,u_1)-A(t,u_2)=0, \  u(0)=0.  $$
Let $g^n\in W^{1,1}, n\in \mathcal{N}$ be the kernel associated with the Yosida approximation of the operator $ \partial_t^\beta$ (cf.\cite{VZ08,VZ15}).
Then by \cite[Theorem 2.1]{VZ08} we have
\begin{equation}\label{2.1}
 \begin{split}
  &  \frac{d}{dt}(g^n \ast \|u(\cdot)\|_H^2)(t) \le 2 \<  \frac{d}{dt}(g^n \ast u)(t) , u(t)\>_H \\
  = & 2~ { }_{V^*}\<  \frac{d}{dt}(g_{1-\beta} \ast u)(t) , u(t)\>_V + 2~  { }_{V^*}\<  \frac{d}{dt}(g^n \ast u)(t)- \frac{d}{dt}(g_{1-\beta} \ast u)(t) , u(t)\>_V \\
  = & 2~ { }_{V^*}\<  A(t,u_1(t)) -A(t,u_2(t)), u(t)\>_V + 2~  { }_{V^*}\<  \frac{d}{dt}(g^n \ast u)(t)- \frac{d}{dt}(g_{1-\beta} \ast u)(t) , u(t)\>_V \\
  \le &  2~  { }_{V^*}\<  \frac{d}{dt}(g^n \ast u)(t)- \frac{d}{dt}(g_{1-\beta} \ast u)(t) , u(t)\>_V=: 2h^n(t).
\end{split}
\end{equation}
Note that $h^n\rightarrow 0$ in $L^1([0, T])$, which yields that $g_\beta \ast h^n \rightarrow 0$ in $L^1([0, T])$ (cf.~\cite{VZ08}).
Moreover, we have
$$  g_\beta \ast  \frac{d}{dt}(g^n \ast \|u(\cdot)\|_H^2)=  \frac{d}{dt}(g^n \ast g_\beta \ast \|u(\cdot)\|_H^2)  \rightarrow
 \frac{d}{dt}(g_{1-\beta} \ast g_\beta \ast \|u(\cdot)\|_H^2) = \|u(\cdot)\|_H^2 $$
 in   $L^1([0, T])$ as $n\rightarrow\infty$.

Hence, by applying convolution with the kernel $g_\beta$ to (\ref{2.1}) we have
$$   \|u(t)\|_H^2 \le 0,  \ a.e. t\in [0, T],  $$
which implies that $u_1=u_2$, i.e. the solution to   (\ref{1.1}) is unique.  \qed

Suppose that $U$ is a Hilbert space and
$W(t)$  is  a $U$-valued cylindrical Wiener process  defined on a
filtered probability space $(\Omega,\mathcal{F},\mathcal{F}_t,\mathbb{P})$.
Now we consider stochastic nonlinear evolution equations with fractional time derivative of  type
\begin{equation}\label{SPDE}
 \partial_t^{\beta}(X(t)-x)+A(t,X(t))=\partial_t^\gamma \int_0^t B(s) dW(s),  \ 0<t<T,  % \ u(0)=u_0\in H,
\end{equation}
where $\gamma \in (0,1]$, $B: [0,T]\rightarrow L_{HS}(U; H)$ is measurable (here $(L_{HS}(U; H), \|\cdot\|_{HS})$
denotes the space of all Hilbert-Schmidt operators from $U$ to $H$).

We want to explain the reason why this form of random noise term is used in (\ref{SPDE}). In fact, heuristically
 \begin{equation}\label{SPDE1}
 \partial_t^\gamma \int_0^t B(s) dW(s)=I_t^{1-\gamma}[B(t)dW(t)]
\end{equation}
which can be used to model systems subject to classical random noise (the case $\gamma=1$) or random noise with certain
memory effects (the case $\gamma< 1$). This form of noise will naturally appear if $e.g.$ you consider the heat transfer
with random effects with memory (see \cite{CKK}) or use the time fractional Duhamel's principle to derive the appropriate form of
stochastic time-fractional diffusion equations (see \cite{MN15} for the case $\gamma=\beta$). The case $\gamma=\beta$ is investigated
in various papers about stochastic heat type or wave type equations, see e.g. \cite{C14,CHN,FN15,MN15,MN16}, and the case $\gamma=1$ is
studied in \cite{CHHH,HH}.

In this work we consider the general case $\gamma<\beta+\frac{1}{2}$, this assumption is natural since the stochastic integral term in (\ref{SPDE1}) is well-defined if and only if ($e.g.$ assuming $B(\cdot)$ is bounded)
$$  \int_0^t (t-s)^{2(\beta-\gamma)}\|B(s)\|_{HS}^2 ds \le C  \int_0^t (t-s)^{2(\beta-\gamma)} ds < \infty, $$
which is equivalent to $\gamma<\beta+\frac{1}{2}$. The same assumption $\gamma<\beta+\frac{1}{2}$ is also used in \cite{CKK} which is deduced there by a differentiability  argument.

For a concise formulation of our result we define for $t\in [0, \infty)$
$$   F(t):= \frac{1}{\Gamma(1+\beta-\gamma)} \int_0^t (t-s)^{\beta-\gamma}B(s) dW(s). $$
\begin{thm}\label{Th2}
Suppose that $\gamma\in(0,1]$, $T\in [0,\infty)$ and that $A$ satisfies $(H1)$-$(H4)$,
$B\in L^\infty([0,T], L_{HS}(U; H))$ if $\gamma<\beta+\frac{1}{2}$ or $B\in L^2([0,T], L_{HS}(U; H))$ if $\gamma\le \beta$.
Assume also that $F\in V, dt\otimes \mathbb{P}$-$a.e.$ (which is $e.g.$ the case if $B(t)$ is a Radonifying map from $U$ to $V$).
Then for every $x\in V$ (\ref{SPDE}) has a unique $\mathcal{F}_t$-adapted solution $X$ such that $X-F-x\varphi \in \mathcal F$, $\mathbb P$-a.s.~for
every $\varphi \in L^\alpha([0,\infty) ; \R)$ with $\varphi \equiv 1$ on $[0,T+1)$. In particular,
$$  X-F-x\varphi \in L^\alpha([0,\infty); V); \   \partial_t^\beta (X-F-x\varphi)\in L^{\frac{\alpha}{\alpha-1}}([0,\infty); V^*), \  \mathbb{P}\text{-}a.s  $$
and for $dt$-a.e.~$t \in [0, T]$,
\begin{equation*}
 X(t)= x - \frac{1}{\Gamma(\beta)}\int_0^t (t-s)^{\beta-1} A(s,X(s))\, ds +  \frac{1}{\Gamma(1+\beta-\gamma)} \int_0^t (t-s)^{\beta-\gamma} B(s) d W(s),
 \ \mathbb{P}\text{-}a.s.
\end{equation*}

Furthermore, $t \mapsto \frac{1}{\Gamma(1-\beta)} \int_0^t (t-s)^{-\beta} (X(s) - x \varphi(s)) \, ds$ and $F$  have continuous $H$-valued $dt$-versions,
and $t \mapsto X(t)$ has a  continuous $V^*$-valued $dt$-version if $\beta \in (\frac{\alpha-1}{\alpha}, 1)$.
\end{thm}
\begin{proof}
Let $u(t)=X(t)-F(t)$, then $u(t)$ satisfies the following equation
\begin{equation}\label{1.2}
 \partial_t^{\beta}(u(t)-x)+A(t,u(t)+F(t))=0,  \ 0<t<T. % \ u(0)=u_0\in H,
\end{equation}
Define
$$ \tilde{A}(t,u)=A(t,u+F(t)), u\in V. $$
Since $F\in  V$ $dt\otimes\mathbb{P}$-$a.e.$, it is easy to see that $\tilde{A}$ still satisfies $(H1)$-$(H4)$.
Hence (\ref{1.2}) has a unique solution $u$, which implies that $u+F$ is the unique solution to (\ref{SPDE}).

The $\mathcal{F}_t$-adaptiveness of the solution follows by the proofs of Theorem \ref{T1} and Lemma \ref{L2}.
Since $g_\alpha\ast g_\beta=g_{\alpha+\beta}$, it easily follows ($e.g.$ by \cite[Proposition 6.3.3]{LR15}) that
$F$ is $\mathbb{P}$-$a.s.$ continuous in $H$. Hence $X$ is continuous in $V^*$ if $\beta \in (\frac{\alpha-1}{\alpha}, 1)$.
\end{proof}

\begin{rem}\label{rem2} From the above proof one can see that Theorem \ref{Th2} also holds for the case that
$A$ and $B$ are random coefficients ($i.e.$ also depend on $\omega\in\Omega$) in a progressively measurable way
(cf. \cite[Section 4.1]{LR15}). It's also obvious that one could take the initial condition $x$ as some $V$-valued random variable.

\end{rem}

\section{Application to Examples}

In this part we will apply our main result to establish the existence and uniqueness of solutions to (stochastic) porous medium equations and $p$-Laplace equations
with time-fractional derivative.
Both are open problems even in the deterministic case. In the recent work \cite{VZ15},  the authors derive very nice decay estimates for the solution of
those equations (by assuming the existence) and
the decay behaviour turns out to be notably different from the classical parabolic case. As said,
here we give a positive answer to the question concerning the existence and
uniqueness of solutions to the time-fractional porous medium equations and $p$-Laplace equations.

Both examples are just model cases which we present here motivated by \cite{VZ15}.
There are many more examples (see e.g.~\cite{BDR16} and  \cite{RRW07}).
In particular, fast diffusion equations are also covered.
We are confident that we can extend our approach to the case of multi-valued operators as well.

Let $\L \subset \R^d$ be an  open bounded domain and $\Delta$ be the Laplace operator, and for $p\in [1,\infty[$ we use $L^p(\L)$ and  $H_0^{n,p}(\L)$ to denote
 the $p$th-integrable Lebesgue space and the Sobolev space  of order $n$ in
 $L^p(\L)$ with \emph{Dirichlet boundary conditions}. Recall that $X^*$ denotes the dual space of a Banach space $X$.

\subsection{Time-fractional porous medium equation}
We first introduce the porous medium operator $A(u):=\Delta\Psi(u)$.  Let $\Psi : \R \to \R$
be a function having the following properties:
\begin{enumerate}
\item [$(\Psi 1)$] $\Psi$ is continuous.

\item [$(\Psi 2)$] For all $s,t\in \R$
$$
   (t-s)(\Psi (t) - \Psi(s)) \geq 0. $$
\item [$(\Psi 3)$] There exist $p\in [2,\infty[,\, c_1\in\, ]0,\infty[,\, c_2\in [0,\infty[$
such that for all $s\in \R$
$$  s\Psi (s) \geq c_1 |s| ^p -c_2. $$
\item [$(\Psi 4)$] There exist $c_3, c_4 \in\, ]0,\infty[$ such that for all $s\in \R$
$$
   |\Psi(s)|\leq  c_3 |s|^{p-1}+c_4,
$$
where $p$ is as in $(\Psi 3)$.
\end{enumerate}

Now we consider the following Gelfand triple
$$   V: = L^p(\L) \subseteq   H:= (H_0^{1,2}(\L))^*  \subseteq   V^*:=(L^p(\L))^*  $$
and  stochastic time-fractional porous medium equation
 \begin{equation}\label{SPME}
  \partial_t^\beta (X(t)-x)=\Delta \Psi (X(t)) + \partial_t^\gamma \int_0^t B(s) \d W(s), \  0<t<T.
 \end{equation}
\begin{rem}
Since here  $H$ is not taken as $L^2(\L)$, by the definition of the Gelfand triple, one should note that
 $$V^*\neq L^{\frac{p}{p-1}}(\L). $$
  In fact, one can prove that (see \cite[Lemma 4.1.13]{LR15})
 $$   V^*= \Delta (L^{\frac{p}{p-1}}(\L)).   $$
\end{rem}

\begin{thm}\label{Th3}
Suppose that $\Psi$ satisfies $(\Psi 1)$-$(\Psi4)$. Then Theorem \ref{Th2} applies with $A:= \Delta \Psi (\cdot)$ and $\alpha := p$.
\end{thm}
\begin{proof}
Now we can define the \emph{porous medium operator}
$A:  V \to V^*  $ by
\begin{align}\label{eq:gelfand 19b}
A(u):= \D \Psi (u) ,\quad u\in L^p(\L).
\end{align}
Note that by \cite[Lemma 4.1.13]{LR15} the operator $A$
is well-defined. The conclusion follows directly from Theorem \ref{Th2} by checking (H1)--(H4) hold for $-A$.

In fact, it is well-known that the porous medium operator satisfies the monotonicity and coercivity properties (see e.g.~\cite[Section 4.1]{LR15}).
We include the proof here for the reader's convenience.
\begin{enumerate}
\item [(H1):] Let $u,v, x\in V= L^p(\L)$ and $\l \in \R$. Then by
\begin{align}\label{eq:gelfand 20}
\begin{split}
   { }_{V^*}\<A(u+\l v), x\>_{V} = &  { }_{V^*}\<  \D \Psi (u+ \l v),x \>_{V}\\
   =&- \int_{\L} \Psi (u(\xi) + \l v(\xi) ) x(\xi) \d \xi.
\end{split}
\end{align}
By $(\Psi 4)$ for $|\l|\leq 1$ the integrand in the right-hand side of \eqref{eq:gelfand 20} is bounded by
$$
   [c_4 + c_3 2^{p-2} (|u|^{p-1} + |v| ^{p-1})]|x|
$$
which by H\"older's inequality is in $L^1 (\L)$. So, (H1) follows by ($\Psi1$) and
Lebesgue's dominated convergence theorem.
\item [(H2):] Let $u,v\in V=L^p (\L)$. Then
\begin{align*}
    { }_{V^*}\< A(u) -A(v) , u-v)\>_{V}
   & =  { }_{V^*}\<\D (\Psi (u) -\Psi(v)), u-v\>_{V}\\
   & = -\int_{\L} [\Psi(u(\xi))-\Psi(v(\xi))](u(\xi) - v(\xi)) \d \xi\\
   & \leq 0,
\end{align*}
where we used $(\Psi 2)$ in the last step.
\item [(H3):] Let $v\in L^p(\L)= V$. Then by $(\Psi 3)$
\begin{align*}
    { }_{V^*}\< A(v),v\>_{V} & = - \int_{\L} \Psi(v(\xi)) v(\xi) \d\xi\\
   &\leq \int_{\L}(-c_1|v(\xi)|^p + c_2 ) \d\xi.
\end{align*}
Hence (H3) is satisfied with $\delta:= c_1,\, \alpha= p$ and $g(t)= c_2 |\L|$.
\item [(H4):] Let $v\in L^p (\L)= V$. Then by \cite[Lemma 4.1.13]{LR15}
and $(\Psi 4)$
   \begin{align*}
      \|A(v)\|_{V^*}
      & = \norm{\D \Psi(v)} _{V^*}\\
      & = \norm{\Psi (v) } _{L^{\frac{p}{p-1}}}\\
      & \leq c_4 |\L|^{\frac{p-1}{p}} + c_3
         \left( \int|v(\xi)|^p \d\xi\right)^{\frac{p-1}{p}}\\
      & = c_4 |\L |^{\frac{p-1}{p}} + c_3 \norm{v}_V^{p-1},
   \end{align*}
   so (H4) holds with $\alpha = p$.
\end{enumerate}
\end{proof}

\begin{rem}\label{4.1.15}

 $(i)$  For $p\in [2,\infty[$ and $\Psi (s) := s|s|^{p-2}$ we have
   $$A(v) = \D (v|v|^{p-2}), \, v \in L^p(\L),$$
    which is the non-linear
   operator appearing in the classical porous medium equation (cf. e.g.\cite{Va06,Va07}), $i.e.$
   $$ \frac{\partial u(t)}{\partial t}= \D (u(t)|u(t)|^{p-2}) , \quad u(0,\cdot )=u_0, $$
   whose solution describes the time evolution of the density $u(t)$ of a substance
   in a porous medium. And as a matter of fact, the time-fractional porous medium equation was first introduced by Caputo in \cite{C99}
   to describe the diffusion of fluids in
   porous media with memory.

 $(ii)$ As mentioned before, our results apply to the situation in \cite{RRW07} where $\Delta\Psi(u)$ is replaced by $L\Psi(u)$, here $L$ is the generator
 of a transient Dirichlet form on $L^2(E,\mathcal{E}, \mu)$ for abstract $\sigma$-finite measure spaces $(E,\mathcal{E}, \mu)$,
 so (\ref{SPDE}) takes the form
 $$  \partial_t^\beta (X(t)-x)- L \Psi (X(t)) = \partial_t^\gamma \int_0^t B(s) \d W(s).   $$
 This, in particular, includes ``fractal''($i.e.$ non-local) operators $L$, as $e.g.$
  fractional Laplacian $-(-\Delta)^\alpha$, $\alpha\in(0,1]\cap (0, \frac{d}{2})$, and the underlying domain $\Lambda$ may be unbounded, e.g.~$\Lambda=\R^d$, $d\geq3$.
 $\Psi$ may belong to the more general class described in \cite{RRW07}, but must be monotone.
 In particular, $\Psi(r) := r |r|^{m-1}$ with $m \in (0,1]$ is covered, $i.e.$ the fast diffusion equation.
So, equations as
$$  \partial_t^\beta (X(t)-x) +(-\Delta)^\alpha\left(X(t)|X(t)|^{m-1} \right)  = \partial_t^\gamma \int_0^t B(s) \d W(s)  $$
 for $m\in (0, \infty)$ are covered.
   \end{rem}

\subsection{Time-fractional $p$-Laplace equation}

   Now we consider stochastic time-fractional  $p$-Laplace equations ($p\ge 2$)
   \begin{equation}\label{SPLE}
    \partial_t^\beta (X(t)-x)= \text{div}  \left( |\nabla X(t)|^{p-2} \nabla X(t)\right) + \partial_t^\gamma \int_0^t B(s)\d W(s),  \  0<t<T.
   \end{equation}
We will choose the follwing Gelfand triple
   $$V:=H_0^{1,p}(\L) \subseteq H:= L^2 (\L) \subseteq V^*= (H_0^{1,p}(\L))^*$$
   and define
   $A:H_0^{1,p}(\L)\to H_0^{1,p}(\L)^*$ by
  \begin{equation}\label{eq:3.4'}
     A(u):= \text{div} (|\nabla u|^{p-2}\nabla u), \; u\in H_0^{1,p}(\L);
   \end{equation}
   more precisely, given $u\in H_0^{1,p}(\L)$, then we define
   \begin{align}\label{eq:gelfand 11}
      { }_{V^*}\< A(u),v\>_{V} := - \int_{\L} |\nabla u(\xi)|^{p-2}
      \< \nabla u(\xi), \nabla v(\xi)\> \d\xi\quad\text{for all } v\in H_0^{1,p}(\L).
   \end{align}
   Here $A$ is called the \emph{p-Laplacian},
   \index{p-Laplacian@$p$-Laplacian}
   also denoted by $\D_p$. \label{symbol32}
   Note that $\D_2=\Delta$.

   It is well-known that $A: V\to V^*$ is well-defined (see e.g.~\cite[Section 4.1]{LR15}). In fact, we only need to show that the right-hand
   side of \eqref{eq:gelfand 11} defines a linear functional in $v\in V$
   which is continuous with respect to $\norm{\;}_V= \norm{\; }_{1,p}$.
   First we recall that  $\nabla u\in L^p (\L;\R^d)$
   for all $u\in H_0^{1,p}(\L)$. Hence by H\"older's inequality
   \begin{align*}
      \int|\nabla u(\xi) |^{p-1} |\nabla v(\xi) | \d\xi&\leq \left(\int |\nabla u (\xi) |^p \d\xi\right)^{\frac{p-1}{p}}\left(\int|\nabla v(\xi) |^p \d\xi\right)^{\frac{1}{p}}\\
      &\leq \norm{u}_{1,p}^{p-1}\norm{v}_{1,p}.
   \end{align*}
   Since this dominates the absolute value of the right-hand side of \eqref{eq:gelfand 11} for all
   $u\in H_0^{1,p}(\L)$,  we have that $A (u)$ is a well-defined element of $(H_0^{1,p}(\L))^*$
   and that
   \begin{align}\label{eq:gelfand 12}
      \norm{A(u)}_{V^*}\leq \norm{u}_V^{p-1}.
   \end{align}

   \begin{thm}\label{Th4}
Suppose that $p\ge 2$. Then Theorem \ref{Th2} applies with $A$ as defined in \eqref{eq:3.4'} and $\alpha := p$.
\end{thm}
\begin{proof}
   The conclusion follows directly from Theorem \ref{Th2} by checking (H1)-(H4) hold for $-A$ here.
   Now we include the proof here also for the reader's convenience.
   \begin{enumerate}
      \item[(H1):] Let $u,v,x\in H_0^{1,p}(\L)$, then by \eqref{eq:gelfand 11} we
      have to show for $\l \in \R, \, |\l |\leq 1$
      \begin{equation}
      \begin{split}
         &\lim_{\l \to 0} \int_{\L} \Big(
         |\nabla (u+ \l v) (\xi) |^{p-2}\< \nabla (u+\l v) (\xi) ,
          \nabla x(\xi)\> \\
    &\quad
          -|\nabla u (\xi)|^{p-2} \< \nabla u(\xi) , \nabla x (\xi)\>
          \Big)
          \d \xi=0.
      \end{split}
      \end{equation}
      Since obviously the integrands converge to zero as $\l \to 0 \, \d\xi$-a.e., we
      only have to find a dominating function to apply Lebesgue's dominated convergence
      theorem. But obviously, since $|\l| \leq 1$
      \begin{align*}
         &|\nabla (u+\l v) (\xi) |^{p-2}| \<\nabla (u+\l v) (\xi) , \nabla x(\xi)\>|\\
         \leq &2^{p-2} \left(|\nabla u(\xi) |^{p-1}+ |\nabla v(\xi)|^{p-1}\right)
         |\nabla x(\xi)|
      \end{align*}
      and the right-hand side is in $L^1(\L)$ by H\"older's inequality as we have seen above.
   \item [(H2):] Let $u,v\in H_0^{1,p}(\L)$. Then by \eqref{eq:gelfand 11}
      \begin{align*}
         &-{ }_{V^*}\<A (u)-A(v), u-v\>_{V}\\
   =& \int_{\L} \<|\nabla u(\xi) |^{p-2} \nabla
         u(\xi) - |\nabla v(\xi) | ^{p-2} \nabla v(\xi) , \nabla u(\xi) - \nabla v(\xi) |>
         \d\xi\\
         =& \int_{\L} (|\nabla u(\xi) |^p + |\nabla v(\xi) |^p - |\nabla u(\xi)|^{p-2}
         \<\nabla u(\xi) , \nabla v(\xi)\>\\
         &\quad - |\nabla v(\xi)|^{p-2} \<\nabla u(\xi) , \nabla v(\xi) \> )\d\xi\\
         \geq& \int_{\L} (|\nabla u(\xi) |^p + |\nabla v(\xi) |^p - |\nabla u(\xi)|^{p-1}
         |\nabla v(\xi)| \\
         &\quad- |\nabla v(\xi)| ^{p-1} |\nabla u(\xi) | )\d\xi\\
         =& \int_{\L} (|\nabla u(\xi) |^{p-1} - |\nabla v(\xi) |^{p-1})
         (|\nabla u(\xi)| - |\nabla v(\xi)|)\d\xi\\
         \geq \,& 0,
      \end{align*}
      since the map $\R_+\ni s\mapsto s^{p-1}$ is increasing. Hence (H2) is proved.
   \item [(H3):] Because $\L$ is bounded by Poincar\'e's inequality
      there exists a constant $c=c(p,d,|\L|)\in ]0,\infty[$ such that
      \begin{align}\label{eq:gelfand 12b}
         \int_{\L}|\nabla u(\xi) |^p \d\xi \geq c \int_{\L} |u(\xi) |^p \d\xi
         \quad\text{for all } u\in H_0^{1,p}(\L).
      \end{align}
Hence by \eqref{eq:gelfand 11} for all $u\in H_0^{1,p}(\L)$
      $$
         { }_{V^*}\<A(u),u\>_{V} = -\int_{\L} |\nabla u(\xi)|^p \d\xi \leq -\frac{\min(1,c)}{ 2}
         \norm{u} _{1,p}^p.
    $$
      So, (H3) holds with $\alpha = p$.
      \item [(H4):] This condition holds for $A$ by \eqref{eq:gelfand 12} with $\alpha=p$.
\end{enumerate}
\end{proof}

\begin{rem}\label{rem:gelfand 1}
$(i)$ It is easy to show that the above result also holds for the more general case where $\Delta_p$ is replaced by
$$ A(u):= \text{div}  \left( \Psi(|\nabla X(t)|) \nabla X(t)\right),   $$
where $\Psi$ satisfies $(\Psi1)$-$(\Psi4)$ with some $p>1$. Moreover, one can further generalize to more general quasilinear differential operator
$$ A(u):=\sum_{|i|\le m}  (-1)^{|i|} D_i A_i(t,x,Du(t,x)), \ \text{where} \   Du=(D_j u)_{|j|\le m}.    $$
Under certain assumptions the above operator also satisfies (H1)-(H4) (cf.~e.g.~\cite[Proposition 30.10]{Z90}).

$(ii)$ In the case of the $p$-Laplacian or more general operators,  it is possible to add some monomials up to order $p$ as pertubation. For example,
   since $H_0^{1,p}(\L) \subset L^p(\L)$ is continuous and
   dense, so
   $$
      A(u):= \text{div} (|\nabla u|^{p-2} \nabla u) - u|u|^{p-2} ,\, u\in H_0^{1,p}(\L),
   $$
    still satisfies (H1)-(H4)  with respect to the Gelfand triple
   $$
      H_0^{1,p}(\L) \subset L^2 (\L) \subset (H_0^{1,p}(\L))^*.
   $$
\end{rem}

 \section*{Acknowledgements}
 The second named author would like to thank his hosts at Madeira University for a very pleasant stay in summer 2017 where
 part of this work was done. Some helpful comments and suggestions from the referee are also gratefully acknowledged.

% -------------------------------------------------------------------------------------------
%                         bibliography
% -------------------------------------------------------------------------------------------
%\bibliographystyle{amsalpha}
\bibliographystyle{amsplain}
%\bibliographystyle{plain}
%\bibliographystyle{plain}
%\bibliographystyle{development}
%\bibliographystyle{jtbnew}
%\bibliographystyle{neuron}
%\bibliographystyle{cell}
%\bibliographystyle{decsci}

%\bibliography{liu}   % expects file "refs.bib"

%------------------------------------------------------------------------------------------------------------------------
%--------------------------------------      References     ------------------------------------------------
%----------------------------------------------------------------------------------------------------------------------
% \input{bib}

\end{document}